\documentclass[times]{mynlaauth}
\usepackage{amsmath,amssymb,graphics,epsfig,color,cite}
\usepackage[ruled,vlined]{algorithm2e}
\newtheorem{theorem}{Theorem}[section]
\newtheorem{proposition}[theorem]{Proposition}

\theoremstyle{remark}
\newtheorem{remark}[theorem]{Remark}
\theoremstyle{definition}

\numberwithin{equation}{section}
\numberwithin{theorem}{section}

\newcommand{\mc}[1]{{\mathcal #1}}
\newcommand{\bb}[1]{{\mathbb #1}}

\newcommand{\rme}{\mathrm{e}}
\newcommand{\rmi}{\mathrm{i}}

\newcommand{\Lameps}{\Lambda_\varepsilon}
\newcommand{\bx}{\mbox{\boldmath{$x$}}}
\newcommand{\by}{\mbox{\boldmath{$y$}}}
\newcommand{\bzero}{\mbox{\bf 0}}
\begin{document}
\title{Computing Unstructured and Structured Polynomial Pseudospectrum Approximations}
\runningheads{S. Noschese and L. Reichel}{Unstructured and Structured Polynomial 
Pseudospectra}

\author{Silvia Noschese\affil{1}\corrauth and Lothar Reichel\affil{2}}

\address{\affilnum{1}Dipartimento di Matematica, SAPIENZA Universit\`a di Roma,
P.le Aldo Moro 5, 00185 Roma, Italy.\break
\affilnum{2}Department of Mathematical Sciences, Kent State University,
Kent, OH 44242, USA.}

\corraddr{Dipartimento di Matematica, SAPIENZA Universit\`a di Roma,
P.le Aldo Moro 5, 00185 Roma, Italy. E-mail: noschese@mat.uniroma1.it. Work  partially supported by INdAM-GNCS.}

%\subjclass[2000]{65F15}


\begin{abstract}
In many applications it is important to understand the sensitivity of eigenvalues of a 
matrix polynomial to perturbations of the polynomial. The sensitivity commonly is 
described by condition numbers or pseudospectra. However, the computation of pseudospectra
of matrix polynomials is very demanding computationally. This paper describes a new 
approach to computing approximations of pseudospectra of matrix polynomials by using 
rank-one or projected rank-one perturbations. These perturbations are inspired by 
Wilkinson's analysis of eigenvalue sensitivity. This approach allows the approximation of
both structured and unstructured pseudospectra. Computed examples show the method to 
perform much better than a method based on random rank-one perturbations both for the
approximation of structured and unstructured (i.e., standard) polynomial pseudospectra.
\end{abstract}

\maketitle

\section{Introduction}\label{sec1}
In many problems in science and engineering it is important to know the sensitivity of the
eigenvalues of a square matrix to perturbations. The pseudospectrum is an important aid 
for shedding light on the sensitivity. Many properties and applications of the 
pseudospectrum of a matrix are discussed by Trefethen and Embree \cite{TE}; see also
\cite{BGN,BGMN,G06,K10,R06}. However, the computation of pseudospectra is a 
computationally demanding task except for very small matrices. Therefore, the development 
of numerical methods for the efficient computation of pseudospectra of medium-sized 
matrices, or partial pseudospectra of large matrices, has received considerable attention; 
see \cite{BG,MP,NR,PRRWY,Wr02,WT}.

The present paper is concerned with the computation of pseudospectra of matrix polynomials
of the form
\begin{equation}\label{matpol}
P(\lambda)=A_m\lambda^m+A_{m-1}\lambda^{m-1}+\ldots+A_1\lambda+A_0,
\end{equation}
where $\lambda\in\mathbb{C}$ and $A_j\in\mathbb{C}^{n\times n}$, $j=0,\ldots,m$. We will
assume that $\det(A_m)\neq 0$. Then $P$ has $mn$ {\it finite} eigenvalues, i.e., there are 
no eigenvalues at infinity. Matrix polynomials of this kind arise in many applications in 
systems and control theory; see, e.g., \cite{GLR,GL,MMMV}. The case $m=1$ corresponds to 
the generalized eigenvalue problem
\[
A_0\bx = -\lambda A_1\bx,
\]
and the special case $A_1=-I_n$ yields a standard eigenvalue problem. Here and throughout 
this paper $I_n$ denotes the identity matrix of order $n$. In some applications the 
matrices $A_j$ in \eqref{matpol} have a structure that should be respected, such as being 
symmetric, skew-symmetric, banded, Toeplitz, or Hankel.

The sensitivity of the eigenvalues of a matrix polynomial \eqref{matpol} to perturbations 
in the matrices $A_j$ is important in applications. This question therefore has received 
considerable attention; see, e.g., \cite{BLP,G06,HT,Ti,TH} and references therein. 
When the matrices $A_j$ are structured, it is natural to only consider perturbations that 
are similarly structured. 

Define the spectrum of $P$,
\[
\Lambda(P)=\left\{ \lambda \in\mathbb{C}: \det (P(\lambda))=0 \right\}.
\]
Given a set of matrices $\Delta=\{\Delta_0,\ldots,\Delta_m\}$, 
$\Delta_j\in\mathbb{C}^{n\times n}$, and a set of weights 
$\omega=\{ \omega_0,\ldots, \omega_m\}$, $\omega_j\geq 0$ for all $j$, we let the 
class of admissible perturbed matrix polynomials be
\begin{equation}\label{calA}
{\cal A}(P,\varepsilon,\Delta,\omega)=\left\{ 
\sum_{j=0}^{m}(A_j+\varepsilon\Delta_j)\lambda^j:\| \Delta_j\|_F\leq \omega_j,
\;j=0,\ldots,m \right\}. 
\end{equation}
The parameters $\omega_j\geq 0$, $j=0,\dots,m$, determine the maximum norm of the 
perturbation $\Delta_j$ of each matrix $A_j$, where $\|\cdot\|_F$ denotes the Frobenius
norm. For instance, to keep $A_j$ unperturbed, we set $\omega_j=0$. 

One approach to investigate the sensitivity of the spectrum of a matrix polynomial to
admissible perturbations is to compute and plot the $\varepsilon$-pseudospectrum of $P$
for several $\varepsilon$-values, where the $\varepsilon$-pseudospectrum of $P(\lambda)$ 
for $\varepsilon>0$ is defined by
\begin{equation}\label{pseudospect}
\Lambda_{\varepsilon}(P)=\left\{ z\in \Lambda(Q): 
Q\in{\cal A}(P,\varepsilon,\Delta,\omega) \right\}. 
\end{equation}

The computation of a $\varepsilon$-pseudospectrum of a matrix polynomial generally is very
computationally intensive, in fact, it is much more demanding than the computation of the 
$\varepsilon$-pseudospectrum of a single matrix; see Tisseur and Higham \cite{TH} for a 
discussion on several numerical methods including approaches based on using a transfer 
function, random perturbations, and projections to small-scale problems. The computations 
use the companion form of the matrix polynomial $P$. This requires working with matrices 
of order $mn$, whose generalized Schur factorization is computed. Therefore, the 
computational methods can be expensive to apply when $mn$ is fairly large and an 
approximation of the $\varepsilon$-pseudospectrum is determined on a mesh with many 
points. Details and counts of arithmetic floating point operations are provided in 
\cite{TH}.

This paper describes a novel approach to approximate the $\varepsilon$-pseudospectra of 
$P$ by choosing particular rank-one perturbations of the matrices $A_j$ (or projected 
rank-one perturbations in case $A_j$ has a structure that is to be respected). The use of
these rank-one perturbations yields approximations of the $\varepsilon$-pseudospectrum 
\eqref{pseudospect} for a lower computational cost than the computation of the 
$\varepsilon$-pseudospectrum. Our approach is inspired by Wilkinson's analysis of 
eigenvalue perturbation of a single matrix; see \cite{W65}. It generalizes an approach 
recently developed in \cite{NR} for the efficient computation of structured or 
unstructured pseudospectra of a single matrix.

This paper is organized as follows. Section \ref{sec2} reviews results on the sensitivity
of a simple eigenvalue of a matrix polynomial, pseudospectra and the distance from 
defectivity for matrix polynomials is considered in Section \ref{sec3}, while the
corresponding discussions for structured perturbations can be found in Sections \ref{sec4}
and \ref{sec5}. Algorithms for computing approximate structured and unstructured 
pseudospectra for matrix polynomials are described in Section \ref{sec6}, and a few
computed examples are presented in Section \ref{sec7}. Finally, Section \ref{sec8} 
contains concluding remarks.

\section{The condition number of a simple eigenvalue of a matrix polynomial}\label{sec2}
Consider the matrix polynomial \eqref{matpol} and assume that the determinant of the 
leading coefficient matrix, $A_m$, is nonvanishing. Let $\lambda_0\in\mathbb{C}$ be an 
eigenvalue of $P$. Then the linear system of equations $P(\lambda_0)\bx=\bzero$ has a 
nonzero solution $\bx_0\in\mathbb{C}^n$ (a right eigenvector), and there is a nonzero 
vector $\by_0\in\mathbb{C}^n$ such that $\by_0^HP(\lambda_0)=\bzero^H$ (left eigenvector).
Here the superscript $^H$ denotes transposition and complex conjugation. The algebraic 
multiplicity of $\lambda$ is its multiplicity as a zero of the scalar polynomial 
$\det( P(\lambda))$. The algebraic multiplicity is known to be larger than or equal to 
the geometric multiplicity of $\lambda_0$, which is the dimension of the null space of 
$P(\lambda_0)$. The following result by Tisseur \cite[Theorem 5]{Ti} is important for the 
development of our numerical method. We therefore present a proof for completeness. 

\begin{proposition}\label{prop1}
Let $\lambda\in\Lambda(P)$ be a simple eigenvalue, i.e. $\lambda\notin\Lambda(P')$, with 
corresponding right and left eigenvectors $\bx$ and $\by$ of unit Euclidean norm. Here
$P'$ denotes the derivative of $\lambda\rightarrow P(\lambda)$. Then the condition number 
of $\lambda$ is given by 
\begin{equation}\label{uncond}
\kappa(\lambda)=\frac{\omega(|\lambda|)}{|\by^HP'(\lambda)\bx|},
\end{equation}
where $\omega(z)=\omega_mz^m+\ldots+\omega_0$. The maximal perturbations are 
\[
\Delta_{j}= \eta \omega_j\rme^{-\rmi j \arg(\lambda)}\by\bx^H,\qquad j=0,\dots,m\,,
\]
for any unimodular $\eta\in\mathbb C$.
\end{proposition}

\begin{proof}
Differentiating 
$\sum_{j=0}^{m}(A_j+\epsilon\Delta_j)\lambda^j(\varepsilon)\bx(\varepsilon)=0$ with 
respect to $\varepsilon\in\mathbb{C}$ yields
\[
\sum_{j=0}^{m}\Delta_j\lambda^j(\varepsilon)\bx(\varepsilon)+ 
\sum_{j=1}^{m}(A_j+\epsilon\Delta_j)j\lambda^{j-1}(\varepsilon)\lambda'(\varepsilon)
\bx(\varepsilon)+
\sum_{j=0}^{m}(A_j+\epsilon\Delta_j)\lambda^j(\varepsilon)\bx'(\varepsilon)=\bzero.
\]
Setting $\varepsilon=0$, one obtains
\[
\sum_{j=0}^{m}\Delta_j\lambda^j\bx+  
\sum_{j=1}^{m}A_jj\lambda^{j-1}\lambda'(0)\bx+\sum_{j=0}^{m}A_j\lambda^j\bx'(0)=\bzero,
\]
where $\lambda=\lambda(0)$. It follows that 
\[
P'(\lambda)\lambda'(0)\bx=-P(\lambda)\bx'(0)-\sum_{j=0}^{m}\Delta_j\lambda^j\bx.
\]
Applying $\by^H$ to both the right-hand side and left-hand side of this equality yields
\[
\by^HP'(\lambda)\bx\cdot\lambda'(0)=-\by^HP(\lambda)\bx'(0)-
\by^H\sum_{j=0}^{m}\Delta_j\lambda^j\bx,
\]
where we note that $\by^HP'(\lambda)\bx\ne 0$ because $\lambda$ is a simple eigenvalue; 
see \cite[Theorem 3.2]{ACL}. Observing that $\by^HP(\lambda)=\bzero^H$, and dividing by 
$\by^HP'(\lambda)\bx$, one has
\[
\lambda'(0)=-\frac{\by^H\sum_{j=0}^{m}\Delta_j\lambda^j\bx}{\by^HP'(\lambda)\bx}.
\]
Taking absolute values yields
\[
|\lambda'(0)|=\frac{|\by^H(\sum_{j=0}^{m}\Delta_j\lambda^j)\bx|}
{|\by^HP'(\lambda)\bx|}\leq\frac{\omega(|\lambda|)}{|\by^HP'(\lambda)\bx|},
\]
where the inequality follows from the bounds $\|\Delta_j\|_F\leq \omega_j$,  
$j=0,\ldots,m$. Finally, letting the matrix $\Delta_j$ be a rank-one matrix of the form
$\eta\omega_j\rme^{-\rmi j \arg(\lambda)}\by\bx^H$ with unimodular $\eta\in\mathbb{C}$ 
(and therefore of Frobenius norm $\omega_j$) for all $j=0,\ldots,m$ shows the proposition. 
\end{proof}

\begin{remark}
Consider the standard eigenvalue problem with $m=1$, $A_0=A$, and $A_1=-I_n$. Then 
$P(\lambda)=A-\lambda I_n$ and $P'(\lambda)=-I_n$. Setting $\omega_0=1$ and $\omega_1=0$,
Proposition \ref{prop1} yields the standard eigenvalue condition number 
$\kappa(\lambda)=1/|\by^H\bx|$. When instead $A_0=A$ and $A_1=-B$, we obtain 
$P(\lambda)=A-\lambda B$ (and $P'(\lambda)=-B$), and the proposition gives the generalized
eigenvalue condition number $\kappa(\lambda)=(\omega_0+\omega_1|\lambda|)/|\by^HB\bx|$; see
\cite{HH}.
\end{remark}

\begin{remark}
If $n=1$, the polynomial is scalar-valued. Let $\lambda$ be a simple root of $P$. Then the
condition number of $\lambda$ is $\omega(|\lambda|)/|P'(\lambda)|$.
\end{remark}

\section{The $\varepsilon$-pseudospectrum of a matrix polynomial and the distance from 
defectivity}\label{sec3}
The $\varepsilon$-pseudospectrum of $P(\lambda)$ given by \eqref{pseudospect} is bounded 
if and only if $\det(A_m+\varepsilon\Delta_m)\neq 0$ for all $\Delta_m$ such that
$\|\Delta_m\|_F\leq \omega_m$. Therefore the boundedness of $\Lambda_{\varepsilon}(P)$ is
guaranteed if $\varepsilon$ is such that the origin does not belong to the 
$\omega_m\varepsilon$-pseudospectrum of $A_m\in\mathbb{C}^{n\times n}$, which is given by
\[
\Lambda_{\omega_m\varepsilon}(A_m):=\left\{z \in\Lambda(A_m+E),~~
E\in\mathbb{C}^{n\times n},~~\|E\|_F\le\omega_m\varepsilon\right\}.
\]
It is easy to see that, if $\varepsilon$ satisfies the constraint
\begin{equation*}
\varepsilon< \min_{1\le i\le n}
\frac{|\lambda_i(A_m)|}{\widehat{\kappa}(\lambda_i(A_m))\omega_m},
\end{equation*}
then  a first order analysis suggests that no component of 
$\Lambda_{\omega_m\varepsilon}(A_m)$, which is approximately a disk of radius 
$\widehat{\kappa}(\lambda_i(A_m))\omega_m\varepsilon$ centered at $\lambda_i(A_m)$ for 
$\omega_m\varepsilon$ small enough, can contain the origin. The origin is on the border of
the disk centered at $\lambda_i(A_m)$  when 
$|\lambda_i(A_m)|={\widehat{\kappa}(\lambda_i(A_m))\omega_m\varepsilon}$. Here 
$\widehat{\kappa}(\lambda(M))$ denotes the traditional condition number of the eigenvalue 
$\lambda$ of the matrix $M\in\mathbb{C}^{n\times n}$. 

Since by assumption $\det (A_m) \neq0$, the 
$\varepsilon$-pseudospectrum \eqref{pseudospect} has at most $mn$ bounded connected 
components. Any small connected component of the $\varepsilon$-pseudospectrum that 
contains exactly one simple eigenvalue $\lambda_0$ of the matrix polynomial $P$ is 
approximately a disk centered at $\lambda_0$ with radius $\kappa(\lambda_0)\varepsilon$.
A matrix polynomial $Q(\lambda)$ is said to be defective if it has an eigenvalue 
$\hat{\lambda}$, whose algebraic multiplicity is strictly larger than its geometric
multiplicity; see \cite{AAB}. Disjoint components of $\Lambda_{\varepsilon}(P)$ associated
with distinct eigenvalues are, to a first order approximation, disjoint disks if 
$\varepsilon$ is strictly smaller than the distance $\varepsilon_{*}$ from defectivity of 
the matrix polynomial $P(\lambda)$, where
\[
\varepsilon_{*}=\inf \{ \| P(\lambda) - Q(\lambda) \|_F \colon 
Q(\lambda) \in \mathbb{C}^{n\times n} \ \mbox{is defective} \}.
\]
%Let machine epsilon, $\varepsilon_M$, be much smaller than $\varepsilon_{*}$, i.e., 
%$0<\varepsilon_M\ll\varepsilon_{*}$. Then the component of $\Lambda_{\varepsilon_M}(P)$ 
%that contains the eigenvalue $\lambda $ is approximately a disk of radius
%$\kappa(\lambda)\varepsilon_M$ centered at $\lambda$. 

A rough estimate of $\varepsilon_{*}$ is 
given by
\begin{equation}\label{rout_1}
\varepsilon:=\min_{\substack{1\le i\le mn\\ 1\le j\le mn \\ j\neq i}}
\frac{|\lambda_i -\lambda_j|}{\kappa(\lambda_i)+\kappa(\lambda_j)}\,.
\end{equation}
The disk centered at $\lambda_i$ is tangential to the disk centered at $\lambda_j$ when 
$|\lambda_i -\lambda_j|=(\kappa(\lambda_i)+\kappa(\lambda_j))\,\varepsilon$. Let 
the index pair $\{\hat\imath,\hat\jmath\}$ minimize the ratio \eqref{rout_1} over all 
distinct eigenvalue pairs. 
%Then the disks with centers $\lambda_{\hat\imath}$ and 
%$\lambda_{\hat\jmath}$ will coalesce first when increasing their radii. 
We will refer
to the eigenvalues $\lambda_{\hat\imath}$ and $\lambda_{\hat\jmath}$ as the {\it most 
$\Lameps$-sensitive pair of eigenvalues}. We note that typically the most 
$\Lameps$-sensitive pair of eigenvalues are not the eigenvalues with the largest
condition numbers.

\section{The structured condition number of a simple eigenvalue of a matrix polynomial}
\label{sec4}
We briefly comment on structured eigenvalue condition numbers for a single matrix before 
turning to matrix polynomials. Consider the set 
${\mc S}\,{\scriptscriptstyle{\substack{\subset \\ \ne}}}\,\mathbb{C}^{n \times n}$ of
structured matrices. For instance, the set may consist of symmetric, tridiagonal, or
Toeplitz matrices. We are concerned with structured perturbations in ${\mc S}$. Let 
$M|_{\mc S}$ denote the matrix in $\mc S$ closest to $M\in\mathbb{C}^{n\times n}$ with 
respect to the Frobenius norm. This projection is used in definition of the eigenvalue 
condition number for structured perturbations, see \cite{BN,NP06,NP07,NR}, where it is 
shown that the eigenvalue condition number for structured perturbations is smaller than 
the eigenvalue condition number for unstructured perturbations. We also will use the 
normalized projection
\[
M|_{\widehat{\mc S}} := \frac{M|_\mc S}{\|M|_\mc S\|_F}
\]
in the definition of maximal structured perturbations in Proposition \ref{prop2} below.

Matrix polynomials \eqref{matpol} are defined by $m+1$ matrices $A_j$, some or all of 
which may have a structure that is important for the application at hand. We refer to a
matrix polynomial with at least one structured matrix $A_j$ as a structured matrix
polynomial. To measure the sensitivity of the eigenvalues of a structured matrix 
polynomial to similarly structured perturbations, we proceed as follows. Let ${\mc S}_j$ 
be a set of structured matrices that the matrix $A_j$ of the matrix polynomial $P$ belongs 
to. If $A_j$ has no particular structure, then ${\mc S}_j={\mathbb C}^{n\times n}$. 
Introduce the set of sets of structured matrices 
${\mc S}=\left\{{\mc S}_0,{\mc S}_1,\dots {\mc S}_m \right\}$ and let 
the class of admissible perturbed matrix polynomials be 
\begin{equation*}
{\cal A}^{\mc S}(P,\varepsilon,\omega,\Delta)=
\left\{ \sum_{j=0}^{m}(A_j+\varepsilon\Delta_j)\lambda^j:\Delta_j\in {\cal S}_j,\;
\| \Delta_j\|_F\leq \omega_j,\;j=0,\ldots,m \right\}. 
\end{equation*}

\begin{proposition}\label{prop2}
Let $\lambda\in\Lambda(P)$ be a simple eigenvalue with corresponding right and left 
eigenvectors $\bx$ and $\by$ of unit Euclidean norm. Then the structured condition number 
of $\lambda$ is given by 
\begin{equation}\label{strcond}
\kappa^{\mc S}(\lambda)=\frac{\omega^{\mc S}(|\lambda|)}{|\by^HP'(\lambda)\bx|},
\end{equation}
where 
\[
\omega^{\mc S}(z)=\sum_{j=0}^m \|\by\bx^H|_{{\mc S}_j}\|_F\,\omega_j z^j \,.  
\]
The maximal perturbations are given by 
\[
\Delta^{\mc S}_{j}=\eta\omega_j\rme^{-\rmi j\arg(\lambda)}\by\bx^H|_{\widehat{\mc S}_j},
\qquad j=0,\dots,m,
\]
for any unimodular $\eta\in\bb C$.
\end{proposition}

\begin{proof}
Differentiating 
$\sum_{j=0}^{m}(A_j+\epsilon\Delta_j)\lambda^j(\varepsilon)\bx(\varepsilon)=\bzero$ with 
respect to $\varepsilon$, as in the proof of Proposition \ref{prop1}, one obtains
\[
|\lambda'(0)|=\frac{|\by^H(\sum_{j=0}^{m}\Delta_j\lambda^j)\bx|}{|\by^HP'(\lambda)\bx|}=
\frac{|\sum_{j=0}^{m}(\by^H\Delta_j\bx)\lambda^j|}{|\by^HP'(\lambda)\bx|},
\]
where $\Delta_j \in {\mc S}_j$ satisfies $\|\Delta_j\|_F\leq \omega_j$, $j=0,\ldots,m$. 
Substituting $\Delta_j$, for $j=0,\ldots,m$, by the structured matrix 
$\eta \omega_j\by\bx^H|_{\widehat{\mc S}_j}\in {\mc S}_j$ with Frobenius norm $\omega_j$,
the upper bound $\omega_j\,\|\by\bx^H|_{{\mc S}_j}\|_F$ for $|\by^H\Delta_j\bx|$ is 
attained. Finally, letting 
$\Delta_j=\eta\omega_j\rme^{-\rmi j\arg(\lambda)}\by\bx^H|_{\widehat{\mc S}_j}$ for all 
$j=0,\ldots,m$ gives
\[
|\lambda'(0)|=\frac{\omega^{\mc S}(|\lambda|)}{|\by^HP'(\lambda)\bx|}.
\]
This concludes the proof.
\end{proof}

\begin{remark}\label{masspr}
The structured condition number \eqref{strcond} is bounded above by the (unstructured)
condition number \eqref{uncond}. In fact, the former can be much smaller than the latter.
For instance, let us consider the quadratic eigenvalue problem $P(\lambda)\bx=\mathbf{0}$,
with $\bx \neq\mathbf{0}$, where 
\begin{equation*}
P(\lambda)=M\lambda^2+C\lambda+K,
\end{equation*}
with the same structured mass matrix $M$, damping matrix $C$ and stiffness matrix $K$ as 
in \cite[Section 4.2]{TH}, i.e., $M:=I_n$, $C:=10\,\mathrm{tridiag}(-1,3,-1)$, and 
$K:=5\,\mathrm{tridiag}(-1,3,-1)$. The $2n$ eigenvalues of the polynomial matrix are real
and negative. In more detail, the spectrum is split into two sets: $n$ eigenvalues are 
spread approximately uniformly in the interval $[-50,-10]$ and $n$ eigenvalues are 
clustered at $-0.5$. Figure \ref{Fig0} shows the unstructured (i.e., standard) condition
numbers (top graph) and structured condiition numbers (bottom graph) for each eigenvalue.
The unstructured condition numbers are seen to be much larger than the structured 
condition numbers.

\begin{figure}[ht]
\begin{center}
\includegraphics[width=7cm]{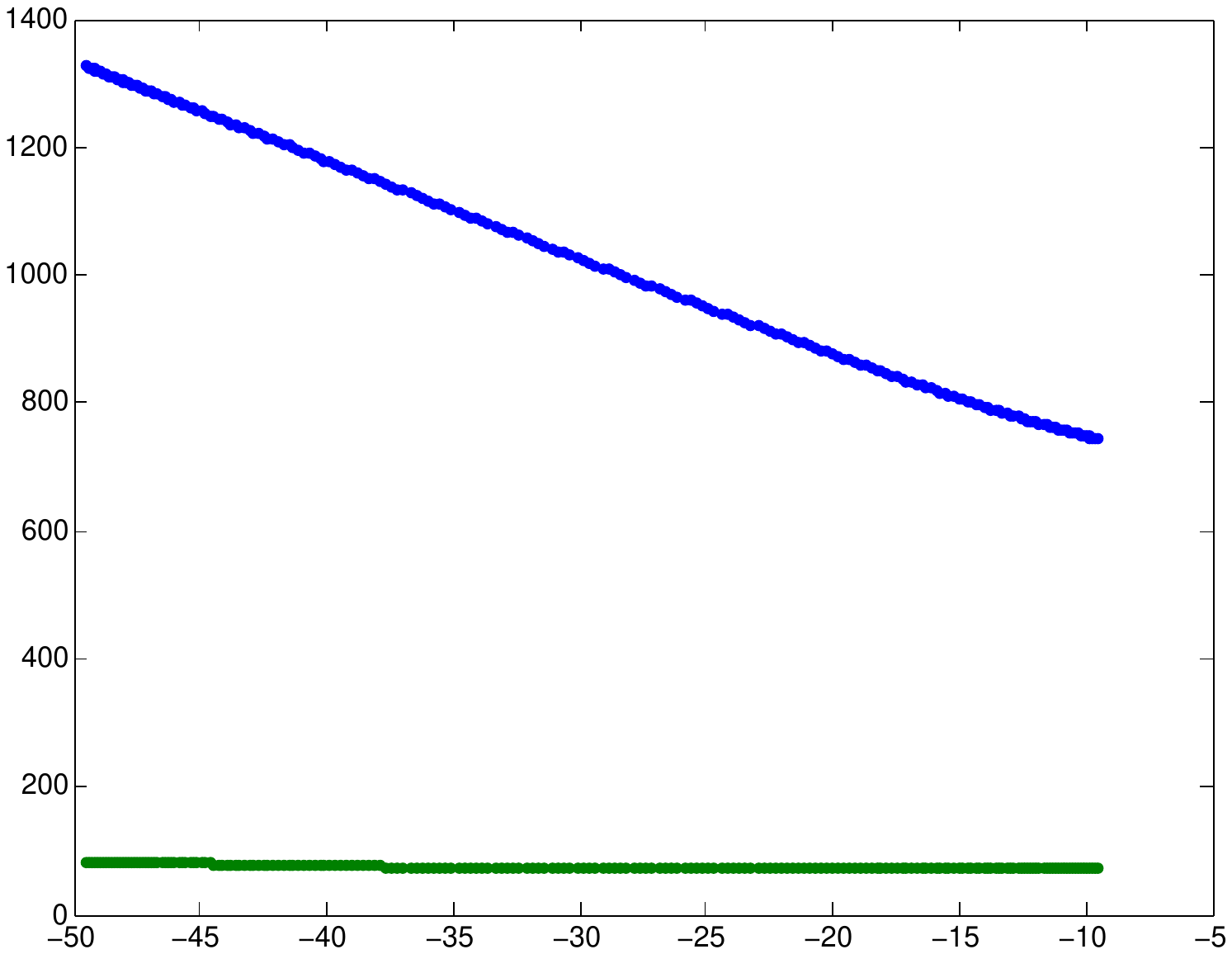} \hskip 0.25cm 
\includegraphics[width=7cm]{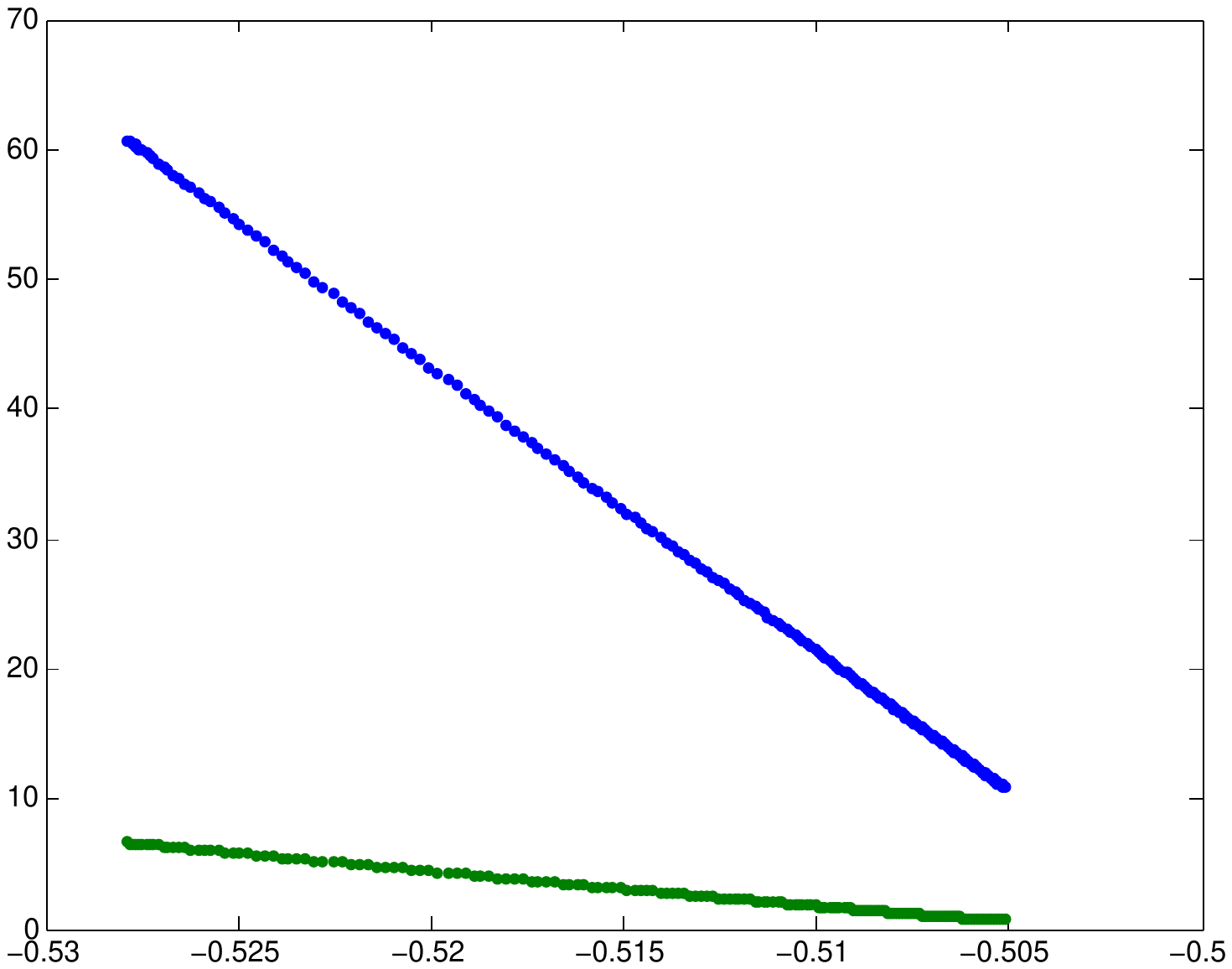}
\caption{Unstructured and structured condition numbers of the eigenvalues of a damped 
mass-spring system with $250$ degrees of freedom considered in \cite[Section 4.2]{TH}.
Left picture: the top graph shows the unstructured condition number and the bottom 
graph the structured condition number versus the eigenvalue values for the leftmost 
subset of eigenvalues. Right picture: Similar as the left picture for the rightmost
subset of eigenvalues. \hfill\break}\label{Fig0}
\end{center}
\end{figure}
\end{remark}

\section{The structured $\varepsilon$-pseudospectrum of a matrix polynomial and the 
structured distance from defectivity}\label{sec5}
The ${\mc S}$-structured $\varepsilon$-pseudospectrum of $P(\lambda)$ is for 
$\varepsilon>0$ defined by
\begin{equation}\label{strpseudospect}
\Lambda_{\varepsilon}^{\mc S}(P)=\left\{ z\in \Lambda(Q): 
Q\in{\cal A}^{\mc S}(P,\varepsilon,\omega,\Delta) \right\}. 
\end{equation}
One has that $\Lambda_{\varepsilon}^{\mc S}(P)$ is bounded if and only if 
$\det (A_m+\varepsilon\Delta_m) \neq0$ for all $\Delta_m \in {\mc S}_j$ such that
$\|\Delta_m\|_F\leq \omega_m$. Thus, the boundedness of $\Lambda_{\varepsilon}^{\mc S}(P)$
is guaranteed if $\varepsilon$ is such that 
$0\notin \Lambda_{\omega_m\varepsilon}^{\mc S_m}(A_m) $, where 
$\Lambda_{\omega_m\varepsilon}^{\mc S_m}(A_m)$ denotes the structured 
$\omega_m\varepsilon$-pseudospectrum of $A_m\in{\mc S_m}$, which is defined by 
\begin{equation*}
\Lambda_{\omega_m\varepsilon}^{\mc S_m}(A_m):=\left\{z \in\Lambda(A_m+E),~~
E\in{\mc S_m},~~\|E\|_F\le\omega_m\varepsilon\right\}.
\end{equation*}
We will assume that $\varepsilon$ satisfies the constraint
\begin{equation*}
\varepsilon< \min_{1\le i\le n}
\frac{|\lambda_i(A_m)|}{\widehat{\kappa}_{\mc S_m}(\lambda_i(A_m))\omega_m}.
\end{equation*}
Then a first order analysis suggests that no component of 
$\Lambda_{\omega_m\varepsilon}^{\mc S_m}(A_m)$ contains the origin. In fact, when 
$\epsilon>0$ is small, the component that contains the eigenvalue $\lambda_i(A_m)$ of 
$A_m$ is approximately a disk of radius 
$\widehat{\kappa}_{\mc S_m}(\lambda_i(A_m))\omega_m\varepsilon$ centered at 
$\lambda_i(A_m)$. Here $\widehat{\kappa}_{\mc S_m}(\lambda)$ denotes the 
${\mc S_m}$-structured condition number of the eigenvalue $\lambda$ in $\Lambda(M)$, where 
$M$ belongs to  the set ${\mc S_m}$ of structured matrices in $\mathbb{C}^{n\times n}$. 

Any small connected component of $\Lambda_{\varepsilon}^{\mc S}(P)$ that contains exactly 
one simple eigenvalue $\lambda_0\in \Lambda(P)$ is approximately a disk centered at
$\lambda_0$ with radius $\kappa^{\mc S}(\lambda_0)\varepsilon$. Such disks of
$\Lambda_{\varepsilon}^{\mc S}(P)$ for distinct eigenvalues are, to a first order 
approximation, disjoint if $\varepsilon$ is strictly smaller than the structured 
distance $\varepsilon^{\mc S}_{*}$ from defectivity of the matrix polynomial $P(\lambda)$.
This distance is given by
\[
\varepsilon^{\mc S}_{*} = \inf \{ \| P(\lambda) - Q(\lambda) \|_F \colon Q(\lambda) \in 
{\mc S} \ \mbox{is defective} \}\;.
\]

%For $\varepsilon_M\ll\varepsilon_{*}^{{\mc S}}$, the component of 
%$\Lambda_{\varepsilon_M}^{\mc S}(P)$ that contains $\lambda$ is approximately a disk of
%radius $\kappa^{\mc S}(\lambda)\varepsilon_M$. 
A rough estimate  of $\varepsilon_{*}^{{\mc S}}$ is provided by
\begin{equation}\label{rout_2}
\varepsilon^{\mc S}:= \min_{\substack{1\le i\le mn\\ 1\le j\le mn \\ j\neq i}}
\frac{|\lambda_i -\lambda_j|}{\kappa^{\mc S}(\lambda_i)+\kappa^{\mc S}(\lambda_j)}\geq
\varepsilon \,.
\end{equation}

Similarly as in Section \ref{sec3}, the disk centered at $\lambda_i$ is tangential to the
disk centered at $\lambda_j$ when 
$|\lambda_i-\lambda_j|=(\kappa^{\mc S}(\lambda_i)+\kappa^{\mc S}(\lambda_j))\,\varepsilon$. 
Let the index pair $\{\hat\imath,\hat\jmath\}$ minimize the ratio \eqref{rout_2} over all 
distinct eigenvalue pairs. We will refer to the eigenvalues $\lambda_{\hat\imath}$ and 
$\lambda_{\hat\jmath}$ as the {\it most $\Lameps^{\mc S}$-sensitive pair of eigenvalues}. 
We note that usually the most $\Lameps^{\mc S}$-sensitive pair of eigenvalues is not made 
up of the worst conditioned eigenvalues with respect to structured perturbations.

\section{Algorithms}\label{sec6}
This section describes algorithms based on Propositions \ref{prop1} and \ref{prop2}
for computing approximations of unstructured and structured pseudospectra of matrix
polynomials. 

Let $\{\lambda_i,\bx_i,\by_i\}_{i=1}^{mn}$ denote eigen-triplets made up of the 
eigenvalues $\lambda_i$ and associated left and right unit eigenvectors, $\bx_i$ and
$\by_i$, respectively, of the matrix polynomial $P$ defined by \eqref{matpol}. We
will assume the eigenvalues to be distinct. If a matrix polynomial has multiple
eigenvalues, then we can apply the algorithms to the ones of algebraic multiplicity
one. Throughout this section $\rmi=\sqrt{-1}$.

Algorithm \ref{algo1} describes our numerical method for the approximation of the
$\varepsilon$-pseudospectrum of a matrix polynomial $P$ defined by matrices $A_j$, 
$j=0,\ldots,m$, without particular structure. The algorithm first determines an estimate 
$\varepsilon$ of the distance to defectivity \eqref{rout_1} of the matrix polynomial and 
the indices $\hat\imath$ and $\hat\jmath$ of the most $\Lambda_\varepsilon$-sensitive pair
of eigenvalues of $P$. It then computes the rank-one matrices $\Delta_{\hat\imath}$ and 
$\Delta_{\hat\jmath}$ defined in Proposition \ref{prop1} for all (simple) eigenvalues 
$\lambda$ of $P$ and for equidistant values on the unit circle in the complex plane. This 
defines the perturbations of the polynomial $P$ at the eigenvalues $\lambda$. The spectra 
of the perturbations of $P$ so obtained are displayed. This simple approach typically 
provides valuable insight into properties of the $\varepsilon$-pseudospectrum of $P$.

\begin{algorithm}
\DontPrintSemicolon
\KwData{matrix polynomial $P$, eigensystem $\{\lambda_i,\bx_i,\by_i,\; \forall i=1:mn\}$,
weights $\{\omega_h, \; \forall h=0:m\}$}
\KwResult{$\Lambda_{\varepsilon}(P)$ approximated by $2N$ simulations}

\nl compute $\varepsilon$, $\{\hat\imath,\hat\jmath\}$  by \eqref{rout_1}\;
\nl compute  $W_{\hat\imath}(\lambda)=\sum_{h=0}^m \omega_h\rme^{-\rmi h \arg(\lambda_{\hat\imath})}\lambda^h\by_{\hat\imath}\bx_{\hat\imath}^H$\; 

\nl compute $W_{\hat\jmath}(\lambda)=\sum_{h=0}^m \omega_h\rme^{-\rmi h \arg(\lambda_{\hat\jmath})}\lambda^h\by_{\hat\jmath}\bx_{\hat\jmath}^H$  \;
\nl display the spectrum of  $P(\lambda)+\varepsilon \rme^{\rmi\theta_k} W_{\hat\imath}(\lambda)$ for  $\theta_k=2\pi(k-1)/N$, $k=1:N$ \;
\nl display the spectrum of  $P(\lambda)+\varepsilon \rme^{\rmi\theta_k} W_{\hat\jmath}(\lambda)$ for $\theta_k=2\pi(k-1)/N$, $k=1:N$ \;
\caption{Algorithm for computing an approximated pseudospectrum} \label{algo1}
\end{algorithm}

Algorithm \ref{algo2} is an analogue of Algorithm \ref{algo1} for the approximation of the
structured $\varepsilon$-pseudospectrum of a matrix polynomial. The algorithm differs from 
Algorithm \ref{algo1} in that the distance to defectivity in the latter algorithm is 
replaced by the structured distance to defectivity \eqref{rout_2} and the rank-one 
perturbations are replaced by structured rank-one perturbations defined in Proposition 
\ref{prop2}.

\begin{algorithm}
\DontPrintSemicolon
\KwData{matrix polynomial $P$, eigensystem $\{\lambda_i,\bx_i,\by_i,\; \forall i=1:mn\}$,
weights $\{\omega_h, \; \forall h=0:m\}$}
\KwResult{$\Lambda_{\varepsilon^{\mc S}}^{\mc S}(P)$ approximated by $2N$ simulations}

\nl compute $\varepsilon^{\mc S}$, $\{\hat\imath,\hat\jmath\}$   by \eqref{rout_2}\;

\nl compute  $W^{\mc S}_{\hat\imath}(\lambda)=\sum_{h=0}^m \omega_h\rme^{-\rmi h \arg(\lambda_{\hat\imath})}\lambda^h\by_{\hat\imath}\bx_{\hat\imath}^H|_{\widehat{\mc S}_h}$ \;

\nl compute $W^{\mc S}_{\hat\jmath}(\lambda)=\sum_{h=0}^m \omega_h\rme^{-\rmi h \arg(\lambda_{\hat\jmath})}\lambda^h\by_{\hat\jmath}\bx_{\hat\jmath}^H|_{\widehat{\mc S}_h}$  \;

\nl display the spectrum of  $P(\lambda)+\varepsilon^{\mc S}\rme^{\rmi\theta_k} W^{\mc S}_{\hat\imath}(\lambda)$ for $\theta_k=2\pi(k-1)/N$, $k=1:N$ \;

\nl display the spectrum of  $P(\lambda)+\varepsilon^{\mc S}\rme^{\rmi\theta_k}  W^{\mc S}_{\hat\jmath}(\lambda)$ for $\theta_k=2\pi(k-1)/N$, $k=1:N$ \;
\caption{Algorithm for computing an approximated structured pseudospectrum} \label{algo2}
\end{algorithm}

Both Algorithms \ref{algo1} and \ref{algo2} are easy to implement. The algorithms require 
the computation of the $mn$ eigenvalues of $n\times n$ polynomial matrices. Evaluating 
the spectrum of $2N$ perturbed polynomial matrices is the main computational burden and 
easily can be implemented efficiently on a parallel computer. However, a laptop computer 
was sufficient for the computed examples reported in the following section.

\section{Numerical examples}\label{sec7}
The computations were performed on a MacBook Air laptop computer with a 1.8Ghz CPU and 4 
Gbytes of RAM. All computations were carried out in MATLAB with about $16$ significant 
decimal digits. 

{\bf Example 1}.
Consider the matrix polynomial $P(\lambda)=A_2\lambda^2+A_1\lambda+A_0$, where $A_0$ and
$A_1$ are real $5\times 5$ matrices with normally distributed random entries with zero 
mean and variance, and $A_2$ is a real tridiagonal Toeplitz matrix of the same order with
similarly distributed random diagonal, superdiagonal, and subdiagonal entries. We choose 
the weights $\omega_i=\|A_i\|_F$, $i=0:2$. The 
eigenvalues of $P$ and their standard and structured condition numbers are shown in Table 
\ref{Tab1}. The structured condition numbers can be seen to be smaller than the standard
condition numbers.

The estimate (\ref{rout_1}) of the (unstructured) distance from defectivity 
$\varepsilon_*$ is $\varepsilon_1=0.0127$. It is achieved for the indices $5$ and $7$, as
well as for the indices $4$ and $6$, of the most $\Lameps$-sensitive pairs of eigenvalues. 
The left plot in Figure \ref{fig1} displays the spectrum of matrix polynomials of the form 
$P(\lambda)+\varepsilon_1 \rme^{\rmi\theta_k} W_5(\lambda)$ and 
$P(\lambda)+\varepsilon_1 \rme^{\rmi\theta_k} W_7(\lambda)$ for $\theta_k=2\pi(k-1)/N$, 
$k=1:N$, and $N=5\cdot10^2$. Thus, the spectrum of $10^3$ matrix polynomials are 
determined. Details of the computations are described by Algorithm \ref{algo1}. We recall 
that the ``curves'' surrounding the eigenvalues $\lambda_j$ lie inside the 
$\varepsilon_1$-pseudospectrum of $P$. The figure illustrates that the eigenvalues 
$\lambda_5$ and $\lambda_7$ might coalesce already for a small perturbation of $P$. 

We remark that since the matrices $A_i$, $i=0:2$, that define the matrix polynomial $P$ 
are real, the eigenvalues of $P$ appear in complex conjugate pairs. The pseudospectrum of
matrix polynomials determined by real matrices is known to be symmetric with respect to the 
real axis in the complex plane. The fact that the left plot of Figure \ref{fig1} is not
symmetric with respect to the imaginary axis depends on that it only shows the spectra
of the matrix polynomials $P(\lambda)+\varepsilon_1 \rme^{\rmi\theta_k} W_5(\lambda)$ and
$P(\lambda)+\varepsilon_1 \rme^{\rmi\theta_k} W_7(\lambda)$ associated with the 
eigenvalues $\lambda_5$ and $\lambda_7$ of $P$, but not of the polynomials
$P(\lambda)+\varepsilon_1 \rme^{\rmi\theta_k} W_4(\lambda)$ and
$P(\lambda)+\varepsilon_1 \rme^{\rmi\theta_k} W_6(\lambda)$ associated with the
eigenvalues $\lambda_4$ and $\lambda_6$. A plot of eigenvalues of all these polynomials is
symmetric with respect to the real axis in the complex plane. 

We compare the approximation of the $\varepsilon_1$-pseudospectrum shown in the left plot
of Figure \ref{fig1} with an approximation of the $\varepsilon_1$-pseudospectrum obtained 
by perturbing $P$ by random rank-one matrices. Specifically, the right plot of Figure 
\ref{fig1} shows the spectrum of matrix polynomials of the form 
$P(\lambda)+\varepsilon_1 \rme^{\rmi\theta_k} E(\lambda)$ with $\theta_k=2\pi(k-1)/N$, 
$k=1:N$, where $N=10^6$, and $E(\lambda)=\sum_{h=0}^m \omega_h\lambda^hR_h$. Here $R_h$
is a rank-one random matrix scaled to have unit Frobenius norm. Despite using $10^6$  
perturbations of $P$, which are many more perturbations than used for producing the 
left plot, the right plot of Figure \ref{fig1} does not indicate that any eigenvalue of 
$P$ might coalesce under small perturbations of the matrix polynomial. This important
property clearly is difficult to detect by using random rank-one perturbations.

Next we turn to structured pseudospectra and perturbations. We obtain from (\ref{rout_2})
the estimate $\varepsilon_2=0.0266$ of the structured distance from defectivity
$\varepsilon^{\mc S}_*$. It is achieved for the eigenvalues $\lambda_8$ and $\lambda_9$.
The left plot in Figure \ref{fig2} displays the spectra of matrix polynomials of the form
$P(\lambda)+\varepsilon_2 \rme^{\rmi\theta_k} W^{\mc S}_8(\lambda)$ and 
$P(\lambda)+\varepsilon_2 \rme^{\rmi\theta_k} W^{\mc S}_9(\lambda)$ with 
$\theta_k=2\pi(k-1)/N$, $k=1:N$, for $N=5\cdot10^2$. The computations are described by
Algorithm \ref{algo2}. The plot shows that the eigenvalues $\lambda_8$ and $\lambda_9$ 
might coalesce under small perturbations of $P$.

The right plot of Figure \ref{fig1} shows the spectrum of matrix polynomials of the form 
$P(\lambda)+\varepsilon_2 \rme^{\rmi\theta_k} E^{\mc S}(\lambda)$ with 
$\theta_k=2\pi(k-1)/N$, $k=1:N$, where $N=10^6$, 
$E^{\mc S}(\lambda)=\sum_{h=0}^m \omega_h\lambda^hR^{\mc S}_h$, and 
$R^{\mc S}_h:=R_h|_{\widehat{\mc S}_h}$ is a unit-norm rank-one random matrix projected 
into ${\mc S}_h$. Despite using $10^6$ perturbations, the plot does not indicate that any
eigenvalues of $P$ might coalesce under small structured perturbations.  $\blacksquare$

\begin{table}[htb!]
\centering
\begin{tabular}{cccc}\hline 
$i$ &$\lambda_i$ & $\kappa(\lambda_i)$ & $\kappa^{\mc S}(\lambda_i)$  \\ 
\hline  
$1$ &$-1.6907$ & $23.2593 $ & $\phantom{1}7.0577 $ \\
$2$ &$-0.9225 + 1.1935\rmi$ & $\phantom{2}5.9741$ & $\phantom{1}1.8875$ \\
$3$ &$-0.9225 - 1.1935\rmi$ & $\phantom{2}5.9741$ & $\phantom{1}1.8875$ \\
$4$ &$\phantom{-}0.5245 + 1.3668\rmi$ & $34.2042$ & $11.5406$ \\
$5$ &$\phantom{-}0.5245 - 1.3668\rmi$ & $34.2042$ & $11.5406$ \\
$6$ &$\phantom{-}0.4113 + 0.7192\rmi$ & $17.4605$ & $\phantom{1}8.3749$ \\
$7$ &$\phantom{-}0.4113 - 0.7192\rmi$ & $17.4605$ & $\phantom{1}8.3749$ \\
$8$ &$\phantom{-}0.6637$ & $18.3210$ & $\phantom{1}9.8822$ \\
$9$ &$\phantom{-}0.2045$ & $\phantom{2}7.4414$ & $\phantom{1}7.3777$ \\
$10$ &$-0.5701$ & $\phantom{2}6.2696$ & $\phantom{1}3.7923$ \\
\hline
\end{tabular}
\caption{Example 1: Eigenvalue condition numbers.}
\label{Tab1}
\end{table}

\begin{figure}[ht]
\begin{center}
\includegraphics[width=7cm]{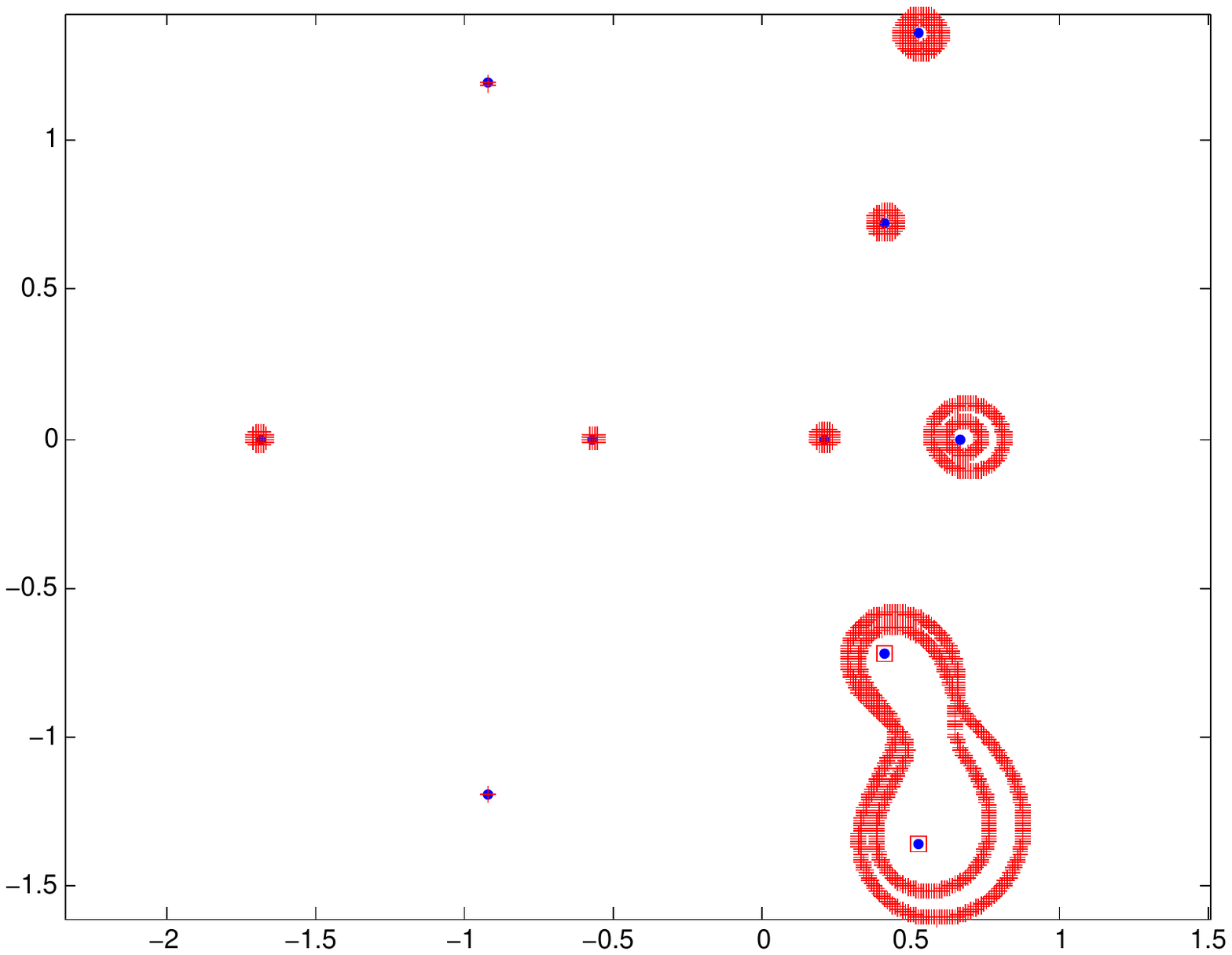}\hskip 0.25cm 
\includegraphics[width=7cm]{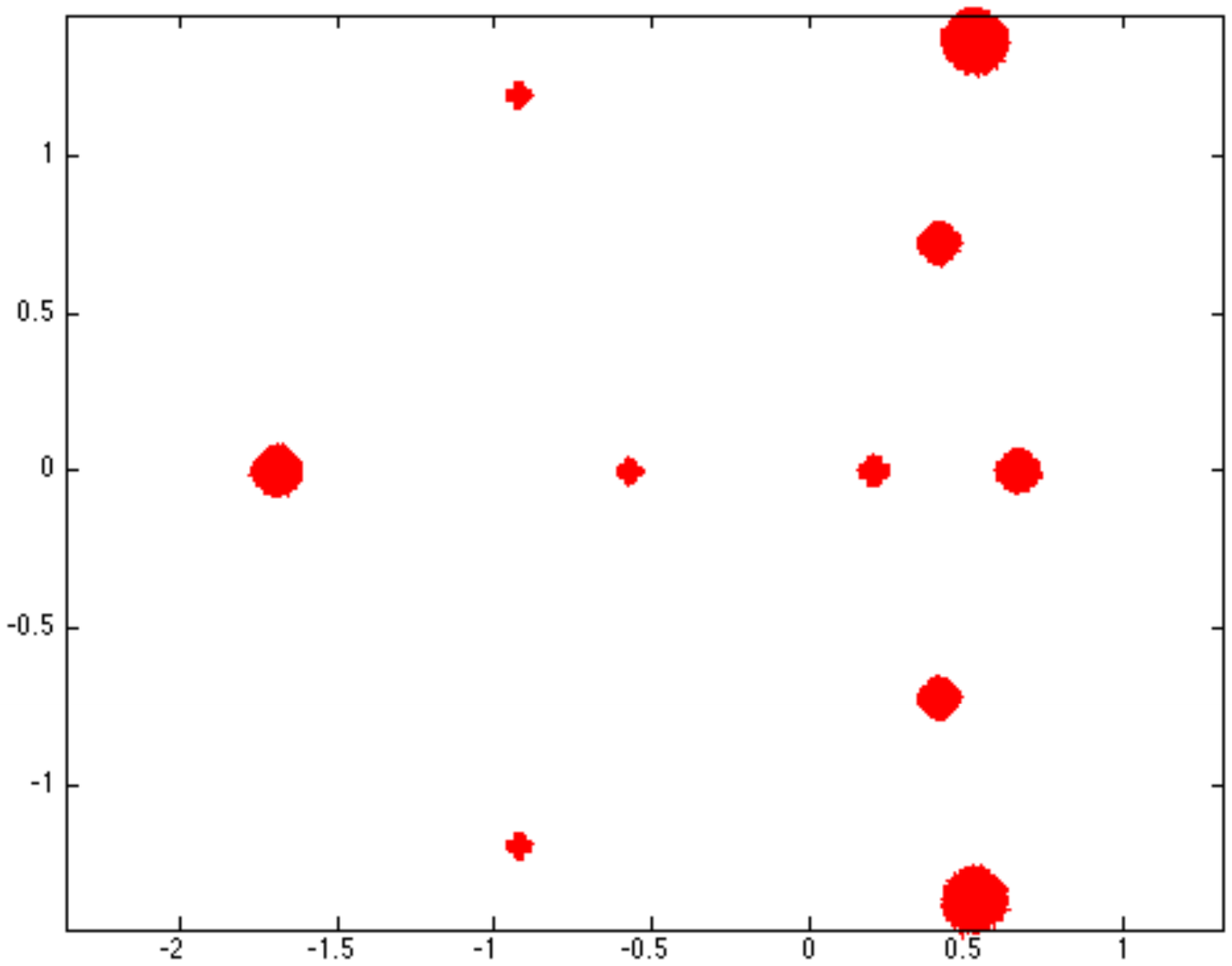}
\end{center}
\caption{Example 1. Left plot: $\Lambda_{\varepsilon_1}(P)$ is approximated by the 
eigenvalues of matrix polynomials of the forms 
$P(\lambda)+\varepsilon_1 \rme^{\rmi\theta_k}W_5(\lambda)$ and 
$P(\lambda)+\varepsilon_1 \rme^{\rmi\theta_k}W_7(\lambda)$, where $\varepsilon_1=0.0127$, 
and the $W_j(\lambda)$ are maximal perturbations associated with the eigenvalues 
$\lambda_j$, $j=5,7$ (marked by red squares), for $\theta_k=2\pi(k-1)/N$, $k=1:N$, with
$N=5\cdot 10^2$. Right plot: $\Lambda_{\varepsilon_1}(P)$ is approximated by the
eigenvalues of matrix polynomials of the form 
$P(\lambda)+\varepsilon_1 \rme^{\rmi\theta_k}E(\lambda)$, where the $E(\lambda)$ are 
random rank-one matrix polynomial perturbations and $k=1:10^6$. \hfill\break}\label{fig1}
\end{figure}

\begin{figure}[ht]
\begin{center}
\includegraphics[width=7cm]{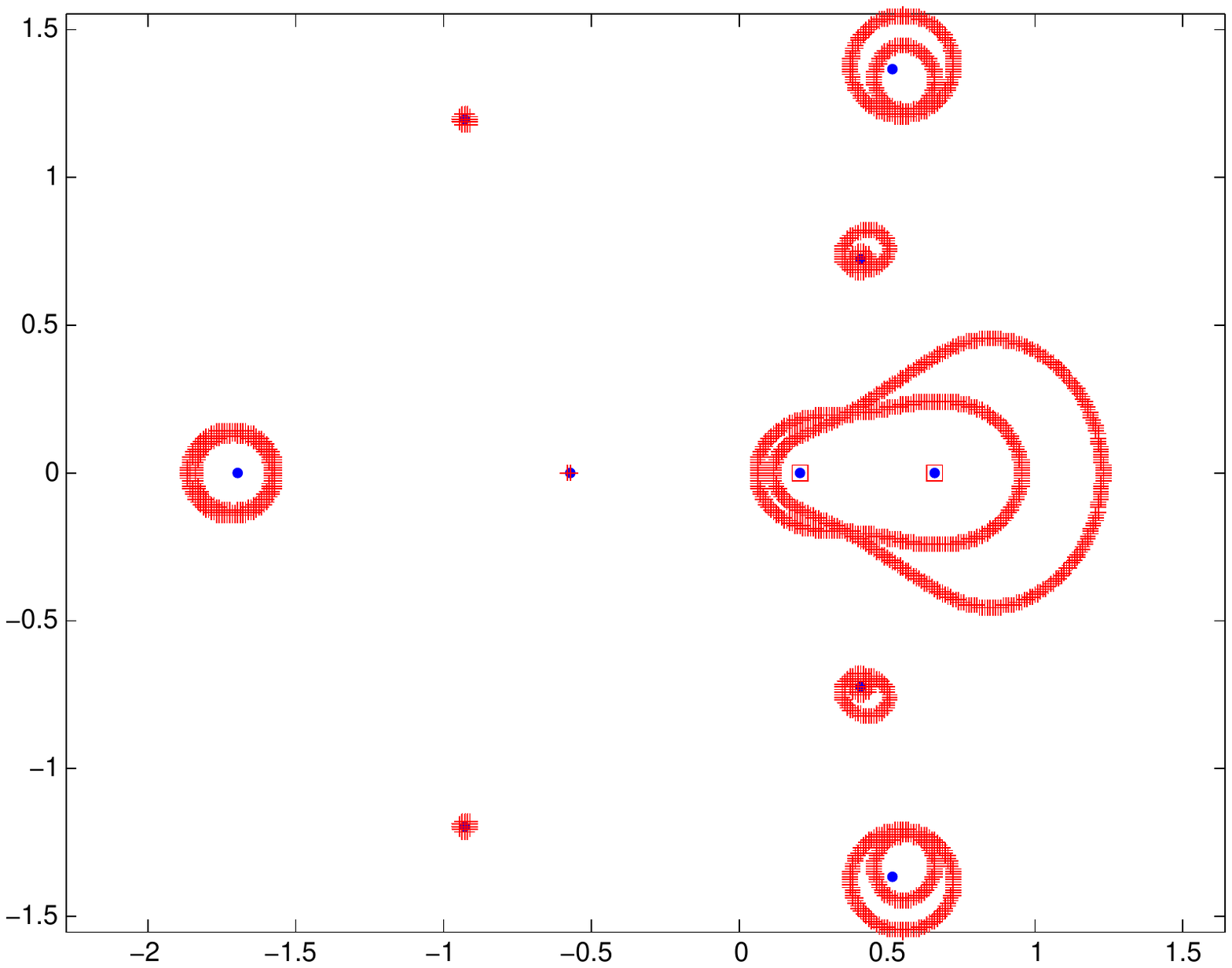}\hskip 0.25cm 
\includegraphics[width=7cm]{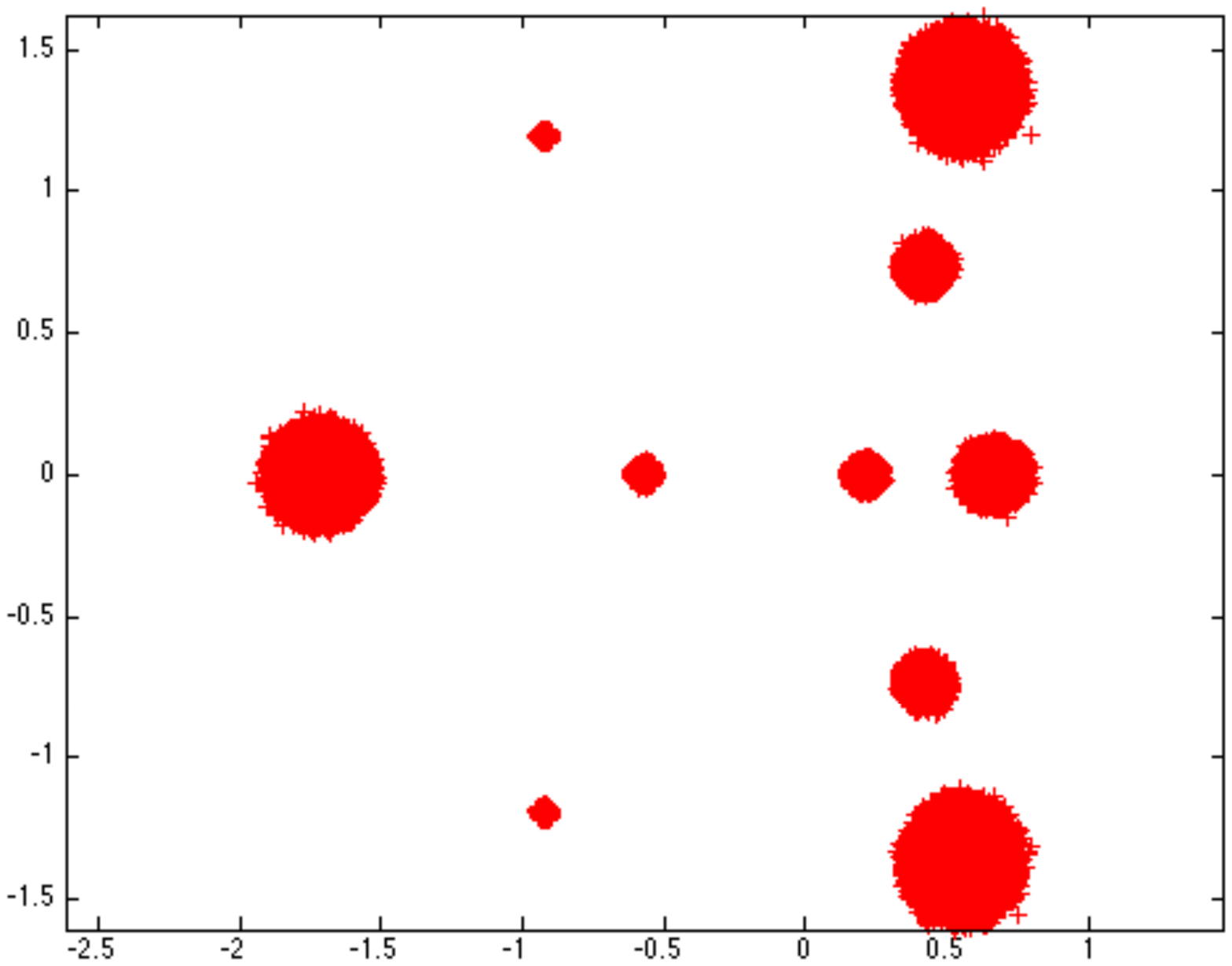}
\end{center}
\caption{Example 1. Left plot: $\Lambda^{\mc S}_{\varepsilon_2}(P)$ is approximated by the
eigenvalues of matrix polynomials of the forms 
$P(\lambda)+\varepsilon_2 \rme^{\rmi\theta_k}W^{\mc S}_8(\lambda)$ and 
$P(\lambda)+\varepsilon_2 \rme^{\rmi\theta_k}W^{\mc S}_9(\lambda)$, where 
$\varepsilon_2=0.0266$, and the $W^{\mc S}_j(\lambda)$ are maximal ${\mc S}$-structured 
perturbations associated with the eigenvalues $\lambda_j$, $j=8,9$ (marked by red a
squares), for $\theta_k=2\pi(k-1)/N$, $k=1:N$, with $N=5\cdot 10^2$. Right plot: 
$\Lambda^{\mc S}_{\varepsilon_2}(P)$ is approximated by the eigenvalues of matrix 
polynomials of the form $P(\lambda)+\varepsilon_2 \rme^{\rmi\theta_k}E^{\mc S}(\lambda)$,
where $E^{\mc S}(\lambda)$ are random ${\mc S}$-structured matrix polynomial 
perturbations, with $k=1:10^6$. \hfill\break}\label{fig2}
\end{figure} 

{\bf Example 2}.
Consider the matrix polynomial $P(\lambda)=A_2\lambda^2+A_1\lambda+A_0$ defined by
$$A_2=\left( \begin{array}{ccc}
17.6 & 1.28 & 2.89 \\
1.28 & 0.824 & 0.413 \\
2.89 & 0.413 & 0.725 \end{array} \right), \qquad\qquad 
A_1=\left( \begin{array}{ccc}
7.66 & 2.45 & 2.1 \\
0.23 & 1.04 & 0.223 \\
0.6 & 0.756 & 0.658 \end{array} \right),$$ 
$$A_0=\left( \begin{array}{ccc}
121 & 18.9 & 15.9 \\
0 & 2.7 & 0.145 \\
11.9 & 3.64 & 15.5 \end{array} \right).$$
This polynomial is discussed in \cite[Section 4.1]{TH}. We choose $\omega=\{1,1,1\}$ 
similarly as in \cite{TH}. The eigenvalues and their condition numbers are shown in Table 
\ref{Tab2}. 

\begin{table}[htb!]
\centering
\begin{tabular}{ccc}\hline 
$i$ &$\lambda_i$ & $\kappa(\lambda_i)$   \\ 
\hline  
$1$ &$-0.8848+8.4415\rmi$ & $27.2147 $  \\
$2$ &$-0.8848 - 8.4415\rmi$ & $27.2147$ \\
$3$ &$\phantom{-}0.0947 + 2.5229\rmi$  & $\phantom{2}0.9276$ \\
$4$ &$\phantom{-}0.0947 - 2.5229\rmi$  & $\phantom{2}0.9276$ \\
$5$ &$-0.9180 +1.7606\rmi$  & $\phantom{2}2.3301$ \\
$6$ &$-0.9180 -1.7606\rmi$ & $\phantom{2}2.3301$ \\
\hline
\end{tabular}
\caption{Example 2: Eigenvalue condition numbers.}
\label{Tab2}
\end{table}

Figure \ref{fig5} displays an approximation of the $\varepsilon$-pseudospectrum of 
$P(\lambda)$ obtained by letting $\varepsilon=10^{-0.8}$ (like in \cite{TH}) and 
computing the eigenvalues of matrix polynomials of the form 
$P(\lambda)+\varepsilon \rme^{\rmi\theta_k} W_1(\lambda)$ for $\theta_k=2\pi(k-1)/10^2$, 
$k=1:10^2$, where $W_1(\lambda)$ is a  maximal perturbation associated with the eigenvalue
$\lambda_1$ (marked by red square) with the largest condition number. Details of the 
computations are described by Algorithm \ref{algo1}. $\blacksquare$

\begin{figure}[ht]
\begin{center}
\includegraphics[width=12cm]{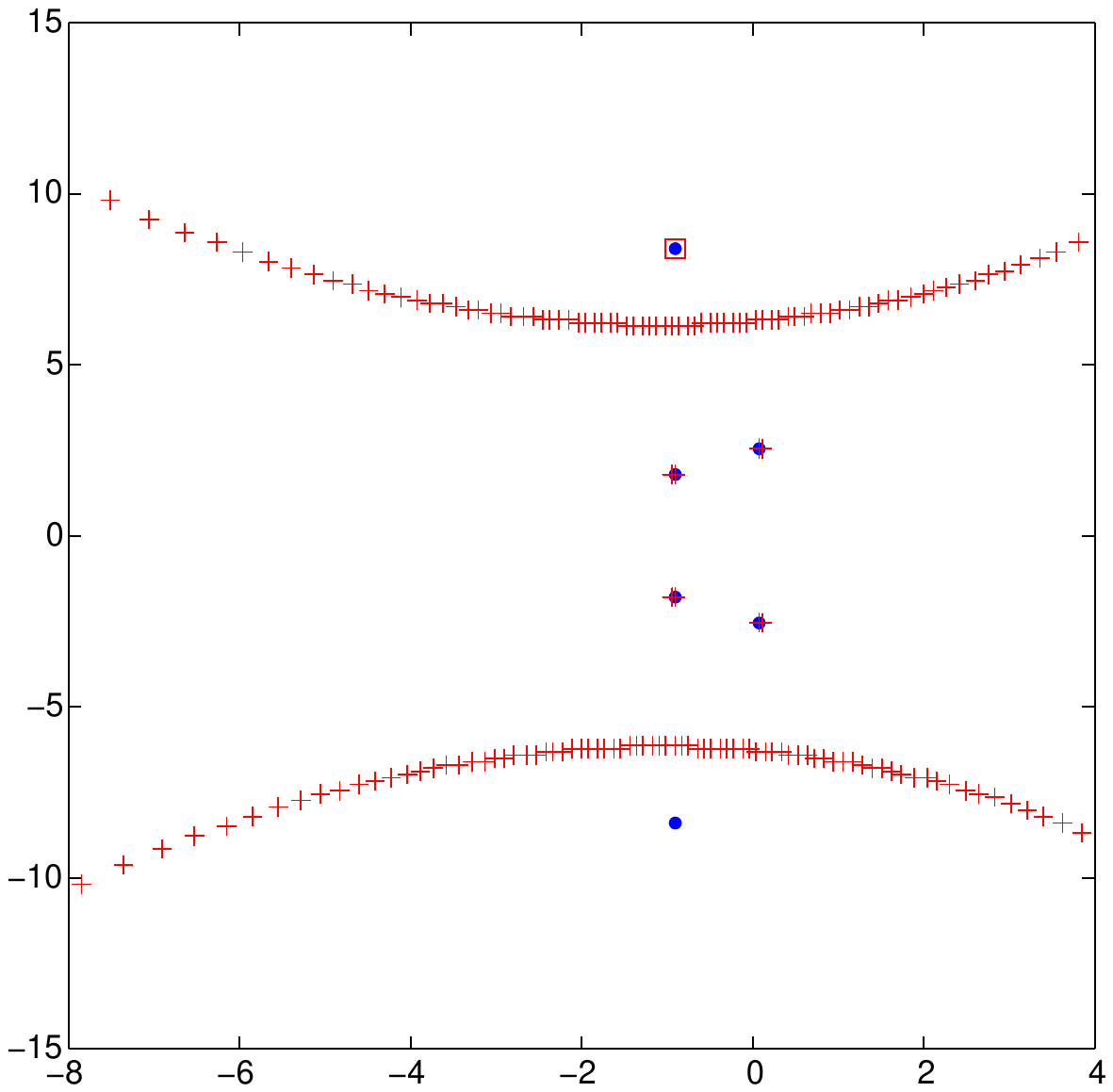}
\end{center}
\caption{Example 2. The pseudospectrum $\Lambda_{\varepsilon}(P)$ for 
$\varepsilon=10^{-0.8}$ is approximated by the eigenvalues of matrix polynomials 
$P(\lambda)+\varepsilon \rme^{\rmi\theta_k} W_1(\lambda)$, where the $W_1(\lambda)$ are
maximal perturbations associated with the eigenvalue $\lambda_1$ (marked by red square),
for $\theta_k:=2\pi(k-1)/10^2$, $k=1:10^2$. \hfill\break}\label{fig5}
\end{figure}

{\bf Example 3}.
We consider the matrix polynomial $P(\lambda)=M\lambda^2+C\lambda+K$ with the structure 
$\cal{S}$ defined in Remark \ref{masspr}. This polynomial is considered in 
\cite[Section 4.2]{TH}. We choose the weights $\omega=\{\|K\|_F, \|C\|_F, \|M\|_F\}$ and
obtain from (\ref{rout_2}) the estimate $\varepsilon_2=3.5709\cdot10^{-7}$ of the 
structured distance from defectivity $\varepsilon^{\mc S}_*$. It is achieved for the 
eigenvalues $\lambda_{493}$ and $\lambda_{494}$. These eigenvalues are the most
$\Lambda_{\varepsilon_2}^{\mc S}$-sensitive pair, but they are not the most ill-conditioned
eigenvalues, despite that their relative distance is only $10^{-6}$.

Figure \ref{fig6} displays the spectra of matrix polynomials of the form
$P(\lambda)+\varepsilon_2 \rme^{\rmi\theta_k}W^{\mc S}_{493}(\lambda)$ and 
$P(\lambda)+\varepsilon_2 \rme^{\rmi\theta_k}W^{\mc S}_{494}(\lambda)$ with 
$\theta_k=2\pi(k-1)/10^2$, $k=1:10^2$. The computations are described by Algorithm 
\ref{algo2}. $\blacksquare$
%eigenvalue 493 ---> -5.051027353396257e-01 1.083775802336067e+01 6.854426275794471e-01
%eigenvalue 494 ---> -5.051032248909045e-01 1.083880340230251e+01 6.855084381315456e-01

\begin{figure}[ht]
\begin{center}
\includegraphics[width=12cm]{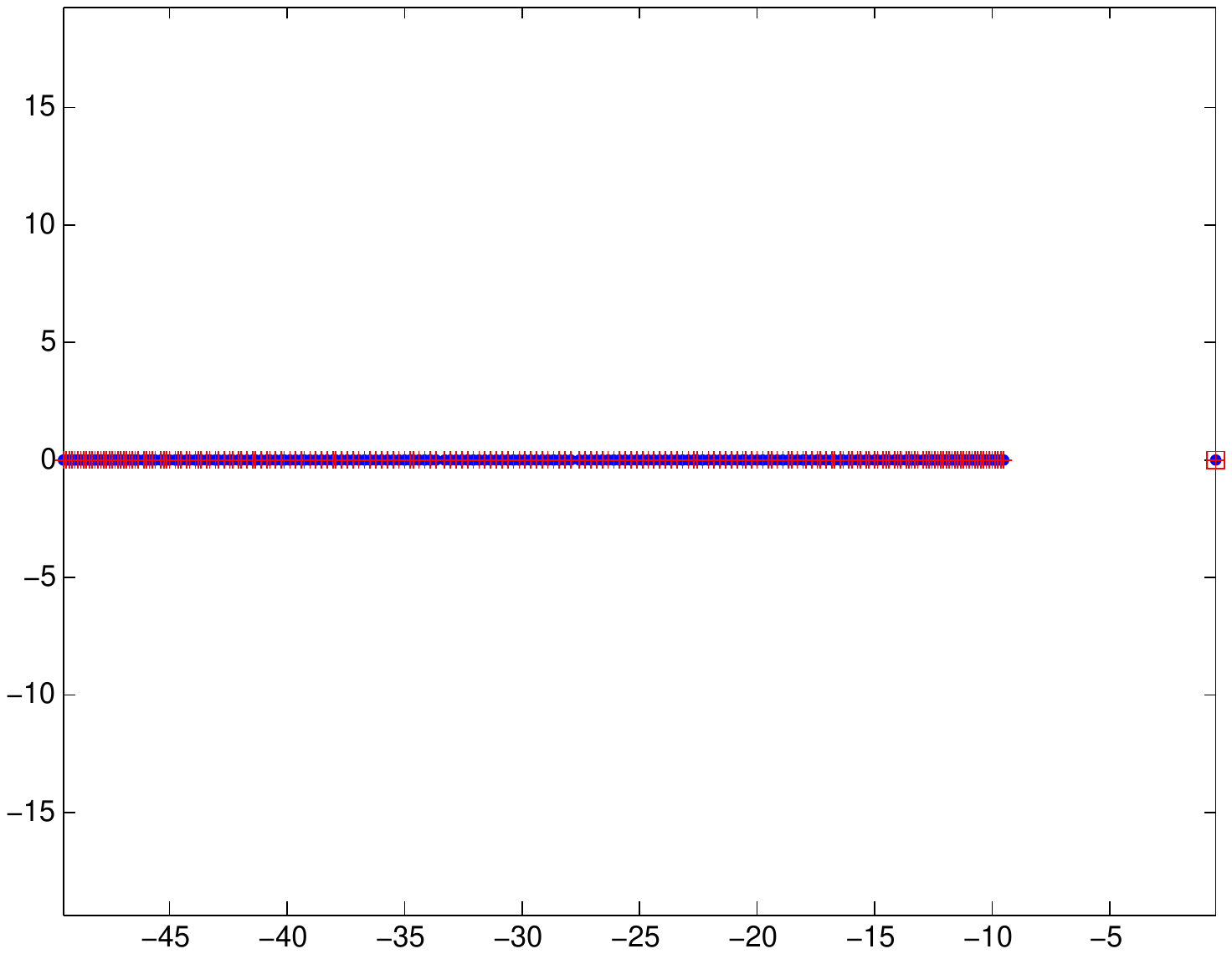}
\end{center}
\caption{Example 3. The structured pseudospectrum $\Lambda^{\mc S}_{\varepsilon_2}(P)$ 
for $\varepsilon_2=3.5709\cdot10^{-7}$ is approximated by the eigenvalues of matrix 
polynomials of the form 
$P(\lambda)+\varepsilon_2 \rme^{\rmi\theta_k}W^{\mc S}_{493}(\lambda)$ and 
$P(\lambda)+\varepsilon_2 \rme^{\rmi\theta_k}W^{\mc S}_{494}(\lambda)$, where the matrices 
$W^{\mc S}_j(\lambda)$ are maximal ${\mc S}$-structured perturbations associated with the 
eigenvalues $\lambda_j$, $j=493,494$ (marked by red squares, though these eigenvalues 
cannot be distinguished in the figure), for $\theta_k=2\pi(k-1)/10^2$, $k=1:10^2$. 
\hfill\break}\label{fig6}
\end{figure}

\section{Conclusions}\label{sec8}
This paper describes a novel and fairly inexpensive approach to determine the sensitivity 
of eigenvalues of a matrix polynomial. Eigenvalues of perturbed matrix polynomials are 
computed, where the perturbations are chosen to shed light on whether eigenvalues of the
given matrix polynomial may coalesce under small perturbations.

\end{document}